\documentclass[12pt,reqno]{amsart}
\usepackage{amsmath,amssymb,amsthm,ifthen,color,amsaddr}
\usepackage[all]{xypic}
\usepackage[mathcal]{euscript}
\usepackage{mathrsfs}

\newtheorem{thm}{Theorem}[section]

\newtheorem{lemma}[thm]{Lemma}
\newtheorem{prop}[thm]{Proposition}
\newtheorem{defin}[thm]{Definition}
\newtheorem{definp}[thm]{Definition/Proposition}

\theoremstyle{remark}
\newtheorem{rem}[thm]{Remark}
\newtheorem{ex}[thm]{Example}
\newenvironment{remark}{\begin{rem}\rm}{\qee\end{rem}}

\newcommand{\Hom}{\operatorname{Hom}}

\newcommand{\cO}{{\mathcal O}}

\newcommand{\PP}{{\mathbb P}}
\newcommand{\AAA}{{\mathbb A}}
\def\co{\mathcal O}
\newcommand{\R}{{\mathbb R}}
\newcommand{\C}{{\mathbb C}}
\newcommand{\Z}{{\mathbb Z}}

\newcommand{\Q}{{\mathbb Q}}

\newcommand{\qee}{\mbox{\hspace{0.2mm}}\hfill$\triangle$}

\newcommand{\var}{{\PP_\Sigma}}
\begin{document}
\begin{flushright} SISSA Preprint 78/2010/fm \\ {\tt arXiv:1011.1003}\end{flushright}
\bigskip
\title[Picard group of  hypersurfaces in   toric 3-folds]{Picard group of  hypersurfaces\\[5pt] in   toric 3-folds}
\author{\small Ugo Bruzzo$^{\P\S\dag}$ and Antonella Grassi$^\P$}
\address{\rm $^\P$ Department of Mathematics, University of Pennsylvania,\\
David Rittenhouse Laboratory, 209 S 33rd Street,\\ Philadelphia, PA 19104, USA\footnote{$^\dag$ On leave of absence from Scuola Internazionale Superiore di Studi Avanzati, Via Bonomea 265, 34136 Trieste, Italy}\footnote{Support for this work was provided by the NSF Research Training Group Grant
DMS-0636606, by {\sc prin} ``Geometria delle variet\`a  algebriche
e dei loro spazi dei moduli '' and  the {\sc infn} project {\sc pi14} ``Nonperturbative dynamics of gauge theories''. U.B. is a member of the {\sc vbac} group.}}
\address{\rm $^\S$ Istituto Nazionale di Fisica Nucleare, Sezione di Trieste}
\thanks{E-mail: {\tt bruzzo@math.upenn.edu, bruzzo@sissa.it, grassi@sas.upenn.edu}}
\date{\today}
\begin{abstract} We show that the usual  sufficient criterion for a very general hypersurface in a smooth projective manifold to have the same Picard number as the ambient variety can be generalized to quasi-smooth hypersurfaces in complete simplicial toric varieties. This sufficient condition always holds for very general K3 surfaces embedded in Fano toric 3-folds.
\end{abstract}

\maketitle \newpage

\section{Introduction}

In this paper we study the Noether-Lefschetz problem for hypersurfaces in complete simplicial toric threefolds, namely, we prove that under a certain condition, a very general hypersurface in an ample linear system in  such a toric threefold $\var$ has the same Picard number as $\var$.
In particular, this holds
for a very general K3 hypersurface in the   anticanonical system of a simplicial toric Fano threefold.
(A property is very general if it holds in the
complement of countably many proper closed subvarieties \cite{Lazarsfeld}.)

This result can be regarded on one hand as a first step towards the  study of Noether-Lefschetz loci of the moduli space of $K3$ hypersurfaces in a  toric Fano threefolds; see also the recent works of \cite{KMPS, MPP,Kloos07}. On the other hand, this completes the picture for computing the Picard number for certain hypersurfaces in the anticanonical system of a toric Fano variety, by handling the unknown case in dimension $3$.

Recall that the Picard number, $\rho(Y)$ of a variety $Y$ is the rank of the N\'{e}ron-Severi group, that is of the image of the Picard group in the second cohomology group with integer coefficients. The Picard group and the Picard number of a toric variety $\var$ can be easily computed from the combinatorial data of $\Sigma$.
Let $X$ be a nondegenerate hypersurface  in the anticanonical system of a simplicial  toric Fano  variety $\var$, with $\dim \var \geq 4$ (note that  a general hypersurface is also nondegenerate).
  In the 80s and 90s it was shown \cite{CoKa99, Bat94, DK,AGM}, that the Picard number of any such $X$ can be explicitly computed from combinatorial data.   This result was a pivotal ingredient in describing the toric version of mirror symmetry (see for example \cite{CoKa99}).
The argument in the above papers is essentially topological and computes the dimension of the second cohomology group of $X$, which happens to be equal to $\rho(X)$ if $\dim (X) \geq 3$, but not necessarily if $\dim (X)=2$. In addition, even the  statement in the above papers   does not hold when $\dim \var =3$, as we see from the case of   Fermat's quartic in $\PP ^3$, which is nondegenerate.

This type of result was generalized  by Roan
to the case of toric varieties (not necessarily Fano) also for the case
when the ambient variety has dimension $d\ge 4$ \cite{Roan96}, and by
Ravindra and Srinivas to general normal varieties, still with the restriction   $d\ge 4$ \cite{RaviSri06}.

This paper then  fills the gap for $\dim \var = 3$.  It was already  known that  $\rho (X)= \rho(\var)$ for particular cases of toric Fano threefolds, namely certain weighted projective spaces \cite{CoxWeighted,SJ,deJongSteen},  as in the higher dimensional case.
 The techniques used in the case of weighted projective spaces are very much tailored to that specific case \cite{Dolgy,SJ,deJongSteen}.
On the other hand, the classical infinitesimal techniques introduced in the 70s by Griffiths, Steenbrink and collaborators to solve the Noether-Lefschetz problem in the smooth case (see for example \cite{CMP03}) cannot be used due to the presence of singularities.
Our argument is partly inspired by Cox's paper \cite{CoxWeighted}: it  generalizes the classical infinitesimal techniques and combines them with more recent results about toric varieties, their Cox ring and their cohomology \cite{BaCox94}. In fact, $X$ and $\var$ are projective orbifolds, and a pure Hodge structure can be defined for them; this will be a key tool in the proof.

In Section \ref{background} we mostly recall some relevant results from \cite{BaCox94}, and   adapt them to the set up of \cite{CMP03}. We start with basic properties of simplicial toric varieties and general hypersurfaces defined by ample divisors.  Moreover we note that the exact sequence defining the primitive cohomology in middle dimension of such a hypersurface splits orthogonally  with respect to the intersection pairing. The middle cohomology is the sum of the primitive cohomology and the ``fixed'' cohomology, i.e., the cohomology inherited from the ambient toric variety; the splitting is consistent with the Hodge decomposition. We then state some results of \cite{BaCox94} which express the primitive cohomology in middle degree in terms of the Jacobian ring of the hypersurface; here we assume that ambient space has odd dimension.

Section \ref{Pic} contains the bulk of the argument: we proceed along the lines of the infinitesimal  arguments of Griffiths for smooth varieties and adapt it to the toric case. We start from the moduli space of quasi-smooth hypersurfaces constructed in \cite{BaCox94}, consider a natural Gauss-Manin connection,   proceed to  prove an infinitesimal Noether-Lefschetz theorem and then the needed global Noether-Lefschetz theorem.
Finally, we focus on the case of $K3$ hypersurfaces in the anticanonical system of a simplicial toric Fano threefold.

The suggestion that a very general hypersurface in a toric Fano threefold $\var$ has the same Picard number as the ambient variety can be found, in a different language, in an unpublished paper by Rohsiepe \cite{Rohsiepe} (see the formula and Remark in the middle of page 3), based on  some dimension counting arguments and trying to generalize to the   case $\dim \var =3 $ a formula that Batyrev proved for $\dim\var = 4$ \cite{Bat94}.

\noindent{\bf Acknowledgements.} We thank Eduardo Cattani, Alberto Collino,   David Cox,
Igor Dolgachev, Luca Migliorini,  Vittorio Perduca, Domingo Toledo  and the referee for useful discussions and suggestions. We are grateful for the hospitality and support offered by the University of Pennsylvania and SISSA.
The first author would also like to thank the staff and the scientists at Penn's Department of Mathematics  for providing an enjoyable and productive atmosphere.

\bigskip

\section{Hypersurfaces in simplicial complete toric varieties}\label{background}
In this section we recall some basic facts about hypersurfaces in toric varieties and their cohomology. We mainly follow the notation in \cite{BaCox94}. All schemes are schemes over the complex numbers.
\subsection{Preliminaries and notation} Let $M$ be a free abelian group of rank $d$, let $N=\Hom(M,\Z)$, and  $N_\R=N\otimes_\Z\R$.
\begin{defin} \cite[Def.~1.1 and 1.3]{BaCox94}  \begin{enumerate} \item A convex subset $\sigma\subset N_\R$ is a rational $k$-dimensional simplicial cone if  there exist $k$ linearly independent primitive elements $e_1,\dots,e_k\in N$ such that
$\sigma = \{\mu_1e_1+\dots+\mu_ke_k\}$,  with $\mu_i$ nonnegative real numbers.
The generators $e_i$ are said to be {\em integral} if for every $i$ and any nonnegative rational number $\mu$, the product $\mu\,e_i$ is in $N$ only if $\mu$ is an integer.

\item Given two rational simplicial cones $\sigma$, $\sigma'$, one says that $\sigma'$ is a face
of $\sigma$ (we then write $\sigma' < \sigma$) if the  set of integral generators of $\sigma'$ is a subset of the set of integral generators of $\sigma$.
\item A finite set $\Sigma=\{\sigma_1,\dots,\sigma_s\}$  of rational simplicial cones is called a
rational simplicial complete $d$-dimensional fan if
\begin{enumerate} \item all faces of cones in $\Sigma$ are in $\Sigma$;
\item if $\sigma,\sigma'\in\Sigma$, then $\sigma\cap\sigma'<\sigma$ and $\sigma\cap\sigma'<\sigma'$;
\item $N_\R= \sigma_1\cup\dots\cup\sigma_s$.
\end{enumerate}\end{enumerate}
\end{defin}

A rational simplicial complete  $d$-dimensional fan  $\Sigma$
 defines a toric variety $\var $ of dimension $d$  having only  Abelian quotient singularities. Moreover, $
\var$ is simply connected, and is  an orbifold.
We shall use the term ``orbifold'' in the following sense (see, e.g., \cite{CoKa99}, Def.~A.2.1): an $n$-dimensional variety $Y$ is an orbifold if every point $y\in Y$ has a neighborhood which is isomorphic to $U/G$ as an analytic space, where $G$ is a subgroup of $Gl_n(\C)$ with no nontrivial complex reflections, and $U$ is  a $G$-invariant neighborhood of the origin of $\C^n$. (A complex reflection is an element in $Gl_n(\C)$ with $n-1$ eigenvalues equal to 1.) A sub-orbifold of an orbifold $Y$ is a subvariety $Y'\subset Y$ with the property that for every $y\in Y'$ there is a local chart $(U/G,0)$ of $Y$ at $y$ such that the inverse image of $Y'$ in $U$ is smooth at 0. Intuitively, a sub-orbifold is a subvariety whose only singularities come from the ambient variety. These notions of orbifold and sub-orbifold are synonymous to those of $V$-manifold and sub-$V$-manifold, which is indeed the terminology used in \cite{BaCox94}. The notion of $V$-manifold is originally due to
 Satake  \cite{Satake}.

\begin{definp} \label{picards}Let  $Cl(\Sigma)$ be the group of Weil divisors 
in $\var$ modulo rational equivalence,
and let $\operatorname{Pic}(\Sigma)$ be the group of line bundles on $\var$ modulo
isomorphism. As the notation suggests, both are intrinsic to the fan $\Sigma$. Both are
finitely generated Abelian groups, and $\operatorname{Pic}(\Sigma)$ is actually free.
Moveover, under our assumptions the toric variety $\var$ is $\Q$-factorial, i.e., the
natural inclusion
$\operatorname{Pic}(\Sigma) \hookrightarrow Cl(\Sigma)$ becomes an isomorphism if one
tensors by $\Q$.
 The rank of the two groups, denoted by $\rho(\Sigma)$,  is also the {Picard number}, the
rank of the N\'eron-Severi group of $\var$.

Recall that the  N\'{e}ron-Severi group of a variety $Y$
 is  the image of the Picard group in the second cohomology group with integer coefficients. 
One can define its rank    as $\rho(Y) \stackrel{def}{=}
\dim_{\Q} NS(Y) \otimes  _{\Z} \Q = \dim_{\Q} H^{2}(Y,\Q) \cap  H^{1,1}(Y,\C)$.
\end{definp}

 The group $\mathbf D(\Sigma)=\operatorname{Spec}\C[Cl(\Sigma)]$ is an affine algebraic group whose
character group is isomorphic to $Cl(\Sigma)$. Since there is a surjection
$\mathbb Z^n \twoheadrightarrow Cl(\Sigma)$, we have an embedding $\mathbf D(\Sigma)\hookrightarrow (\C^\ast)^n$,
and a natural action of $\mathbf D(\Sigma)$ on the affine space $\AAA^n$.
The quotient $\mathbf T(\Sigma)=(\C^\ast)^n/\mathbf D(\Sigma)$ is an algebraic torus.
Below we shall show that this group   acts naturally on $\var$.

\begin{defin} (\cite{Cox95}) Given a fan $\Sigma$, consider a variable $z_i$ for each 1-dimensional cone $\varsigma_i$  in $\Sigma$, and let $S(\Sigma)$ be the polynomial ring $\C[z_1,\dots,z_n]$. For every $\sigma\in\Sigma$, let $z_\sigma = \prod_{\varsigma_i\not\subset \sigma} z_i$, and
let $B(\Sigma)$ the ideal in $S(\Sigma)$ generated by the $z_\sigma$'s.

$S(\Sigma)$ is called the \emph{Cox ring}.
\end{defin}

$S(\Sigma)$ is a graded ring, with grading provided by the class group, $S(\Sigma) = \oplus_{\beta\in Cl(\Sigma) }S_\beta$.
We identify the affine space $\AAA^n$ with $\operatorname{Spec}S(\Sigma)$,
and denote by $Z(\Sigma)$ the affine variety in $\AAA^n$ given by the ideal $B(\Sigma)$. If we set $U(\Sigma)=\AAA^n-Z(\Sigma)$, the group $\mathbf  D(\Sigma)$ acts on $U(\Sigma)$, and the toric variety $\var$ may be represented as $U(\Sigma)/\mathbf  D(\Sigma)$.
This yields an action of $\mathbf  T(\Sigma)$ on $\var$.
For every face $\tau$ in $\Sigma$ we shall denote by
$\mathbf T_\tau\subset \var$ the orbit of $\tau$ in $\var$ under this action.

\subsection{Quasi-smooth hypersurfaces} From now on we assume that $\var$ is projective.
Let $L$ be an ample line bundle on $\var $, and denote by $\beta\in Cl(\Sigma)$
its degree; a section of $L$ is a polynomial in $S_\beta$.

\begin{defin} \cite[Def.~3.1]{BaCox94} Let $f$ be a section of $L$, and
let $\mathbf V(f)$ be the zero locus of $f$ in $\operatorname{Spec}S(\Sigma)$. We say that
the hypersurface $X$ cut  in  $\var $ by the equation $f=0$ is {\em quasi-smooth} if
$\mathbf V(f)$ is smooth outside $Z(\Sigma)$.
\end{defin}

\begin{defin}\label{nondegenerate}\cite[Def.~4.13]{BaCox94}  If $L$ is an ample line bundle
on $\var$, a hypersurface $X$ is said to be {\em nondegenerate} if $X\cap\mathbf T_\tau$ is a smooth 1-codimensional subvariety of
$\mathbf T_\tau$ for all $\tau$ in $\Sigma$.
\end{defin}

\begin{prop}\cite[Prop.~3.5, 4.15]{BaCox94}  If $f$ is the general section of an ample invertible sheaf, then $X$ is nondegenerate.
Moreover, every nondegenerate hypersurface $X\subset\var$ is quasi-smooth.
Thus, if   $f$ is a general section of $L$, its zero locus is a quasi-smooth hypersurface $X$ in $\var $, hence it is an orbifold.
\end{prop}
An important fact is that the complex cohomology of an orbifold has a pure
Hodge structure in each dimension \cite{Stee77,Tu86}.

We also note that in view of the homotopy hyperplane Lefschetz theorem, which holds for the embedding $ X \hookrightarrow \var $
\cite[Thm.~1.2 Part II]{GoMac88}, $X$ is simply connected if $\dim (\var) \geq 3$.

\subsection{Primitive cohomology of a hypersurface} Let $L$ be an ample line bundle on $\var$, and let $X$ be a hypersurface in $\var$ cut by a section $f$ of $L$
(note that by \cite{BaCox94},
Proposition 10.8, $f$ lies in $B(\Sigma)$). Denote by $i\colon X \to \var $ the inclusion, and by $i^\ast\colon H^\bullet( \var ,\C)
\to H^\bullet(X,\C)$ the associated morphism in cohomology; $i^\ast\colon H^{d-1}( \var ,\C)
\to H^{d-1}(X,\C)$ is injective by Proposition 10.8 in \cite{BaCox94}.

\begin{defin}\cite[Def.~10.9]{BaCox94} The primitive cohomology group $PH^{d-1}(X)$ is the quotient
$H^{d-1}(X,\C)/$ $i^\ast(H^{d-1} (\var ,\C))$.
\end{defin}

\begin{lemma} \label{ortho} The exact sequence
$$ 0 \to i^\ast(H^{d-1} (\var ,\C)) \to H^{d-1}(X,\C) \to PH^{d-1}(X) \to 0 $$
splits orthogonally with respect to the intersection pairing in $H^\bullet(X,\C)$. The same is true
with coefficients in $\Q$.
\end{lemma}

\begin{proof} The hard Lefschetz theorem holds also for projective orbifolds  (this follows from the results  in \cite{SaitoKyoto}; a simple proof is given in
\cite{Zaf09}).

Then cupping by $c_1(L)$ we get   an isomorphism
$\ell\colon H^{d-1}(\var,\C) \to H^{d+1}(\var,\C)$. Let $i_\ast \colon H^{d-1}(X,\C)\to
H^{d+1}(\var,\C)$ be the Gysin map. We claim that the following commutative diagram
$$\xymatrix{
&& 0\ar[d] & 0 \ar[d]\\
&0 \ar[d] \ar[r] & H^{d-1}(\var,\C) \ar[d]_{i^\ast} \ar[r] & H^{d-1}(\var,\C) \ar[d]_\ell \ar[r] & 0 \\
0 \ar[r] & \ker i_\ast \ar[d] \ar[r] & H^{d-1}(X,\C) \ar[d] \ar[r]_{i_\ast} &  H^{d+1}(\var,\C) \ar[d]  \ar[r] & 0  \\
0 \ar[r] &  \ker i_\ast  \ar[d] \ar[r]  & PH^{d-1}(X) \ar@<-2mm>@/_10pt/[u]_s \ar[r]\ar[d] & 0 \\
& 0 & 0}$$
 provides a straightforward splitting $s$ of the above exact sequence.  Let $\langle\,,\rangle$ be
the intersection pairing in   cohomology both in $H^\bullet(X,\C)$ and $H^\bullet(\var,\C)$, and recall that  $i^\ast$ and $i_\ast$ are adjoint with respect to the intersection pairing.
The upper-right square commutes since by Poincar\'e duality
$$\langle i_\ast i^\ast\alpha,\beta\rangle = \langle i^\ast\alpha,i^\ast\beta\rangle
=\langle c_1(L) \cup\alpha,\beta\rangle=\langle \ell(\alpha),\beta\rangle\,.$$
If $\alpha\in H^{d-1}(\var,\C)$ and $\beta \in PH^{d-1}(X)$,
we have
$$ \langle i^\ast(\alpha),s(\beta)\rangle=\langle \alpha, i_\ast(s (\beta ))\rangle =0\,.$$
If the statement is true with coefficients in $\C$ it also true with coefficients in $\Q$ since
$H^\bullet(X,\C)\simeq H^\bullet(X,\Q)\otimes_\Q\C$.
\end{proof}

\begin{remark} The kernel of $i_\ast$ in  $H^{d-1}(X,\C)$ is sometimes called the ``variable cohomology'' $H^{d-1}_{\rm var}(X,\C)$;  in degree $d-1$ the variable and primitive cohomologies of $X$
are then isomorphic.
\end{remark}

Both $H^{d-1}( \var ,\C)$ and $ H^{d-1}(X,\C)$ have pure Hodge structures,
and the morphism $i^\ast$ is compatible with them, so that $PH^{d-1}(X)$ inherits
a pure Hodge structure. We shall write
$$ PH^{d-1}(X) = \bigoplus_{p=0}^{d-1} PH^{p,d-1-p}(X).$$

 The following Proposition \ref{iso} implicitly uses a generalization of
Bott's vanishing theorem, called the Bott-Steenbrink-Danilov theorem, which indeed holds under our assumptions. The exact statement is that $H^i(\var,\Omega^p_\var(L))$ $=0$ for all
$i>0$ and $p\ge 0$ if $L$ is an ample line bundle on $\var$.
This was stated without proof by Danilov \cite{Dan78} and proved in \cite{BaCox94} (Theorem 7.1).

\begin{prop} \label{iso} There is a natural isomorphism
$$ PH^{p,d-p-1}(X) \simeq \frac{H^0(\var,\Omega^d_\var((d-p+1)X)}{
H^0(\var,\Omega_\var^d((d-p)X)+dH^0(\var,\Omega_\var^{d-1}((d-p)X)}$$
\end{prop}
\begin{proof} This follows from comparing Corollaries 10.2 and 10.12 in \cite{BaCox94}.
\end{proof}
The resulting projection map, multiplied by the factor $(-1)^{p-1}/(d-p+1)!$, will be denoted
by
\begin{equation}\label{residue} r_p \colon H^0(\var,\Omega^d_\var((d-p+1)X) \to PH^{p,d-p-1}(X)
\end{equation}
and is  called the {\em $p$-th residue map} in analogy with the classical case.

\begin{defin} Let $X$ be any hypersurface in $\var$ cut by a section $f$ of $L$ and  let $J(f)$ be the ideal of the Cox ring generated by the derivatives of $f$.
The ring $R(f)=S(\Sigma)/J(f)$ is the
 {\em Jacobian ring}  of $S(\Sigma)$.
\end{defin}

  The Jacobian ring encodes all the information about the primitive cohomology of $X$:

\begin{prop} If $p\ne d/2-1$, $ PH^{p,d-p-1}(X) \simeq R(f)_{(d-p)\beta-\beta_0 }$,  where
$\beta_0 = - \deg K_{\var }$, $\beta=\deg L$.
\label{isoring}\end{prop}
\begin{proof} \cite{BaCox94} Theorem 10.13. \end{proof}

\bigskip

\section{The Picard group of the general toric threefold}\label{Pic}
\subsection{The Gauss-Manin connection}
Let ${\mathcal Z}$ be the open subscheme of $\vert L \vert$ pa\-ram\-e\-triz\-ing the quasi-smooth hypersurfaces in  $\vert L \vert$, and let $\pi\colon \mathscr F \to {\mathcal Z}$ be the tautological family on ${\mathcal Z}$; we denote by $X_z$  the fiber of $\mathscr F$  at  $z\in \mathcal Z$.
Let $\mathscr H^{d-1}$ be the local system on ${\mathcal Z}$ whose fiber at $z$
is the   cohomology $H^{d-1}(X_z)$, i.e., $\mathscr H^{d-1}=R^{d-1}\pi_\ast\C$.
It defines a flat connection $\nabla$ in the
vector bundle $\mathscr E^{d-1} = \mathscr H^{d-1} \otimes_\C \cO_{\mathcal Z}$, the
{\em Gauss-Manin connection} of $\mathscr E^{d-1}$. Since the hypersurfaces $X_z$ are quasi-smooth, the Hodge structure of the fibres $H^{d-1}(X_z)$ of $\mathscr E^{d-1}$ varies analytically with $z$ \cite{Stee77}. The corresponding filtration defines holomorphic subbundles $F^p\mathscr E^{d-1}$, and the graded object of the filtration defines holomophic bundles $Gr_F^p(\mathscr E^{d-1})$. The bundles
$\mathscr E^{p,d-p-1}$ given by the Hodge decomposition are not holomorphic subbundles of $\mathscr E^{d-1}$, but are diffeomorphic to $Gr_F^p(\mathscr E^{d-1})$, and as such they have a holomorphic structure. The quotient bundles
$\mathscr {PE}^{p,d-p-1}$ of $\mathscr E^{p,d-p-1}$ correspond to the primitive cohomologies of the hypersurfaces $X_z$.
Let $\pi_p: \mathscr E^{d-1}\to \mathscr{PE}^{p,d-p+1}$ be the natural projection.

 We denote by $\tilde\gamma_p$  the cup product
$$ \tilde\gamma_p\colon H^0(\var, \cO_\var(X))   \otimes H^0(\var,\Omega_\var^d((d-p)X))
\to H^0(\var,\Omega_\var^d((d-p+1)X))\,.$$
If $z_0$ is the point in ${\mathcal Z}$ corresponding to $X$, the space $  H^0(\var, \cO_\var(X))/\C(f)$, where $\C(f)$ is the 1-dimensional subspace of $H^0(\var, \cO_\var(X))$ generated by $f$, can be identified with $T_{z_0}{\mathcal Z}$. The morphism $\tilde\gamma_p$ induces in cohomology   the Gauss-Manin connection:

\begin{lemma} Let $\sigma_0$ be a primitive class in $PH^{p,d-p-1}(X)$,
let $v\in T_{z_0}{\mathcal Z}$, and let $\sigma$ be a section of $\mathscr E^{p,d-p-1}$
along a curve in $\mathcal Z$ whose tangent vector at $z_0$ is $v$, such that $\sigma(z_0)=\sigma_0$.

Then   \begin{equation}\label{GM} \pi_{p-1}( \nabla_v(\sigma)) =r_{p-1}( \tilde\gamma_p(\tilde v\otimes \tilde\sigma))
\end{equation}
where  $r_p$, $r_{p-1}$ are the residue morphisms defined in equation \eqref{residue}, $\tilde\sigma$ is an element in $H^0(\var,\Omega_\var^d((d-p+1)X))$
such that $r_p(\tilde\sigma)=\sigma_0$, and $\tilde v$ is a pre-image of $v$ in $H^0(\var, \cO_\var(X))$.

In particular the following diagram commutes:
\begin{equation}\label{commuta}
\xymatrix{
\displaystyle H^0(\var, \cO_\var(X)) 
  \otimes H^0(\var,\Omega_\var^d((d-p)X))
\ar[r]^{\ \ \ \ \ \ \ \ \ \tilde\gamma_p}\ar[d]_{\phi\otimes r_p} &  H^0(\var,\Omega_\var^d((d-p+1)X))
\ar[d]_{r_{p-1}} \\
T_{z_0}{\mathcal Z} \otimes PH^{p,d-1-p}(X) \ar[r]^{\ \ \ \ \ \ \  \gamma_p }& PH^{p-1,d-p}(X) }
\end{equation}
where $\gamma_p$ is the morphism that maps $v\otimes \alpha$ to $\nabla_v\alpha$ and $\phi$ is  the projection
 $\phi\colon$ $  H^0(\var, \cO_\var(X)) \to T_{z_0}{\mathcal Z}$.
\end{lemma}

\begin{proof} This is a standard computation, see \cite{CMP03}, Proposition 5.4.3.
Let $f_i$ be local representatives, with respect to a suitable cover $\{U_i\}$ of $\var$, of the section $f$. Via the isomorphism of Proposition \ref{iso}, we locally represent $\sigma_0$ by the meromorphic differential forms
$\omega_i/f_i^{d-p+1}$. A tangent vector $v\in T_{z_0}{\mathcal Z}$ represents a deformation
$f_i \mapsto f_i+tg_i$ where $t$ is a complex parameter, and $g_i$ are holomorphic functions. Then $\nabla_v(\sigma)$ is represented by
$$\left[\frac{d}{dt} \frac{\omega_i}{(f_i+tg_i)^{d-p+1}}\right]_{t=0} =
-(d-p+1) \frac{g_i\,\omega_i}{f_i^{d-p+2}}\,.$$
But the right-hand side of this equation is, up to a suitable factor, the argument
of the map $r_{p-1}$ in the right-hand side of equation \eqref{GM}.
\end{proof}

\begin{lemma} If $\alpha$ and $\eta$ are sections of $\mathscr E^{p,d-p-1}$ and
$\mathscr E^{d-p,p-1}$ respectively, then for every tangent vector $v\in T_{z_0}{\mathcal Z}$,
\begin{equation}\label{compa}\nabla_v\alpha\cup \eta= -   \alpha \cup \nabla_v\eta\,.
\end{equation}
\end{lemma}
\begin{proof} The Gauss-Manin connection is compatible with the cup product by definition, i.e.,
$$\nabla_v (\alpha\cup\eta) =  \nabla_v\alpha\cup \eta +  \alpha \cup \nabla_v\eta\,. $$
But $\alpha\cup\eta=0$ because it is an element in $\mathscr E^{d,d-2}$.
\end{proof}

\subsection{The moduli space of hypersurfaces in $\var$}
Let ${\operatorname{Aut}}_\beta(\var)$ be the subgroup of $ \operatorname{Aut}(\var) $
which preserves the grading $\beta$. The coarse moduli space $\mathcal M_\beta$ for the general quasi-smooth hypersurfaces in $\var$ with  divisor class ${\beta}$  may be constructed as a quotient
\begin{equation}\label{moduli} U/\widetilde{\operatorname{Aut}}_\beta(\var)\,,
\end{equation} \cite{BaCox94, AGM}, where $U$ is an open subset of
$H^0(\var,\cO_{\var}(X))$, and $\widetilde{\operatorname{Aut}}_\beta(\var)$ is the unique nontrivial extension
$$ 1 \to D(\Sigma) \to \widetilde{\operatorname{Aut}}_\beta(\var) \to {\operatorname{Aut}}_\beta(\var) \to 1\,. $$
By differentiating, we have    a surjective map
$$\kappa_\beta\colon H^0(\var, \cO_\var(X)) \to T_X \mathcal M_{\beta}\,,$$ which is the analogue of the  Kodaira-Spencer map.

The  local system $\mathscr H^{d-1}$
and its various sub-systems do not descend to the moduli space $\mathcal M_{\beta}$, because the group ${\operatorname{Aut}}_\beta(\var)$ is not connected.  Nevertheless, this group
has a connected subgroup
$\operatorname{Aut}^0_\beta(\var)$ of finite order, and, perhaps after suitably shrinking $U$, the quotient $\mathcal M^0_\beta \stackrel{def}{=} U/\operatorname{Aut}^0_\beta(\var)$ is a finite \'etale covering
 of $\mathcal M_\beta$ \cite{CoxDon, AGM}.

\begin{prop} There is a morphism
\begin{equation}\label{gamma2} \gamma_p\colon T_X \mathcal M_{\beta} \otimes PH^{p,d-1-p}(X)\to  PH^{p-1,d-p}(X) \end{equation}
such that the   diagram
\begin{equation*}\label{commuta2}\xymatrix{
H^0(\var, \cO_\var(X))
\otimes H^0(\var,\Omega_\var^d((d-p)X))
\ar[r]^{\ \ \ \ \ \ \ \ \ \cup }\ar[d]_{\kappa_\beta\otimes r_p} &  H^0(\var,\Omega_\var^d((d-p+1)X))
\ar[d]_{r_{p-1}} \\
T_X \mathcal M_{\beta} \otimes PH^{p,d-1-p}(X) \ar[r]^{\ \ \ \ \ \ \  \gamma_p }& PH^{p-1,d-p}(X) }
\end{equation*}
commutes.
\end{prop}
\begin{proof}
It suffices to prove the Proposition with $\mathcal M_\beta$ replaced by $\mathcal M^0_\beta$; in fact the tangent spaces at points $\mathcal M^0_\beta$ are canonically isomorphic to the tangent spaces at the image points in $\mathcal M_\beta$.

If $\rho\colon {\mathcal Z} \to \mathcal M^0_\beta$ is the induced map (where ${\mathcal Z}$ has been suitably restricted),  the local system
$\mathscr H^{d-1}$ descend to a local system $\rho_\ast \mathscr H^{d-1}$ on $\mathcal M^0_\beta$, and $\rho^\ast \rho_\ast \mathscr H^{d-1}\simeq  \mathscr H^{d-1}$ (the natural morphism $\mathscr H^{d-1} \to \rho^\ast \rho_\ast \mathscr H^{d-1}$ is an isomorphism on the stalks  due to topological base change; note that $\rho$ is proper). Thus we obtain on $\mathcal M^0_\beta$
holomorphic bundles that are equipped with a Gauss-Manin connection, which   is trivial in the direction of the fibers of $\rho$. If we define again  $\gamma_p$ by $\gamma_p(v\otimes \alpha)=\nabla_v(\alpha)$ (where $\nabla$ is now the Gauss-Manin connection on $\mathcal M^0_\beta$), the commutavity of the diagram in the statement follows from the commutativity of the diagram \eqref{commuta}.
\end{proof}

The tangent space $T_X\mathcal M_{\beta}$ at a point representing a hypersurface $X$ is naturally isomorphic to the degree ${\beta}$
summand of the Jacobian ring of $f$, that is, $T_X\mathcal M_{\beta}\simeq R(f)_{\beta}$ \cite{BaCox94}. Moreover,  by Proposition \ref{isoring}, $ PH^{p,d-p-1}(X) \simeq R(f)_{(d-p)\beta-\beta_0 }$.

\begin{prop} \label{gammaisring} Under these isomorphisms, the morphism $\gamma_p$ in equation \eqref{gamma2} coincides with the multiplication in the ring $R(f)$,
$$ R(f)_\beta \otimes R(f)_{(d-p)\beta-\beta_0} \to R(f)_{(d-p+1)\beta-\beta_0}\,.$$
\end{prop}
\begin{proof} Theorem 9.7 in 
\cite{BaCox94} implies 
 $$H^0(\var,\Omega_\var^d((d-p)X) / H^0(\var,\Omega_\var^d((d-p-1)X) \simeq S_{(d-p)\beta-\beta_0}\,,$$ 
and, moreover, $H^0(\var,\co_\var(X))$ $\simeq S_\beta$; the cup product corresponds to the product in the ring $S$. This implies that the ``top square'' of the 3-dimensional diagram
$$
\xymatrix@C=-25pt{
\displaystyle {H^0(\var, \cO_\var(X)) \otimes \atop H^0(\var,\Omega_\var^d((d-p)X)) }
 \ar[rr]^\cup\ar[dr]\ar[dd]_{\kappa_\beta\otimes r_p} &&
 H^0(\var,\Omega_\var^d((d-p+1)X))  \ar'[d]_{r_{p-1}}[dd]\ar[dr] \\ &
S_\beta\otimes S_{(d-p)\beta-\beta_0}\ar[dd] \ar[rr] &&
  S_{(d-p+1)\beta-\beta_0}\ar[dd] \\
T_X \mathcal M_{\beta} \otimes PH^{p,d-1-p}(X) \ar'[r][rr]\ar[dr] &&
PH^{p-1,d-p}(X) \ar[dr] \\ &
R(f)_\beta\otimes R(f)_{(d-p)\beta-\beta_0} \ar[rr] &&
R(f)_{(d-p+1)\beta=\beta_0}
}$$
commutes. We need to show that the ``bottom square'' commutes as well, which will follow from the commutativity of the ``side squares'', and the surjectivity of the morphism $\kappa_\beta\otimes r_p$.
The commutativity of the   diagram  on the right is contained in the proof of Theorem 10.6 in \cite{BaCox94}. The commutativity of the   diagram on the left  follows from the commutativity of the previous diagram, with $d-p+1$ replaced by $d-p$, and the commutativity of
$$\xymatrix{
H^0(\var, \cO_\var(X))  \ar[d]\ar[r]^(.7){\sim} & S_\beta\ar[d]  \\ T_X \mathcal M_{\beta} \ar[r]^{\sim} & R(f)_\beta
}$$
which is shown in the proof of Proposition 13.7 in \cite{BaCox94}.
\end{proof}

 \subsection{Picard group}
Our aim is  now to prove the following result. Let us recall that a property is said to be  {\em very general} if it holds in the complement of a countable union of subschemes of positive codimension \cite{Lazarsfeld}. Also recall that the
Picard number $\rho(X)$ is the rank of the N\'eron-Severi group, i.e., $\rho(X) = \dim_\Q  (H^{1,1}(X,\C) \cap H^2(X,\Q))$.

\begin{thm} Let $\var$ be a 3-dimensional complete simplicial toric variety,  $L$ an ample line bundle on $\var$, and $X$ a very general quasi-smooth hypersurface in the linear system $\vert L\vert$.
If  the morphism $\gamma_2\colon T_X\mathcal M_\beta \otimes PH^{2,0} (X)  \to  PH^{1,1} (X)$
is surjective, then   $X$ and $\var $ have the same Picard number.
\label{picard}\end{thm}

Theorem \ref{picard} will follow from two Lemmas. In the first Lemma no restriction on the dimension $d$ of $\var$ needs to be made, in the second we shall assume that $d$ is odd.

The first  Lemma is an  ``infinitesimal Noether-Lefschetz theorem", such as  Theorem 7.5.1 in \cite{CMP03}.

 Denote by $H^{d-1}_T(X) \subset H^{d-1}(X)$ the subspace
of  the cohomology classes that are annihilated by the action of the Gauss-Manin connection. Coefficients may be taken in $\C$ or $\Q$. Note that $H^{d-1}_T(X)$ has a Hodge structure.

\begin{lemma} For a given $p$ with $ 1 \le p \le d-1$,   assume that
the morphism $$\gamma_p\colon T_X \mathcal M_{\beta} \otimes PH^{d-p,p-1}(X)\to  PH^{d-p-1,p}(X) $$ is surjective. Then $H^{p,d-1-p}_T(X)=i^\ast(H^{p,d-1-p}(\var))$.
\end{lemma}

\begin{proof} Replace $ \mathcal M_\beta$ by $\mathcal M^0_\beta$,
and consider the local systems $\mathscr E^{d-1}$ and $\mathscr{PE}^{p,d-p-1}$ on $\mathcal M^0_\beta$.
Take $$\alpha\in H^{p,d-1-p}_T(X)\cap PH^{p,d-1-p}(X).$$
We regard classes in $PH^{p,d-1-p}(X)$ as elements in the fiber of $\mathscr{PE}^{p,d-p-1}$
at the point $[X]\in M^0_\beta$.
By hypothesis $\beta\in PH^{d-p-1,p}(X)$ can be written as
$\beta = \sum_i \gamma_p(t_i\otimes\eta_i)$ with $\eta_i\in PH^{d-p,p-1}(X)$. Then
by equations \eqref{GM} and \eqref{compa}
$$\langle \alpha,\beta\rangle = \sum_i \langle \alpha, \gamma_p(t_i\otimes\eta_i) \rangle
= \sum_i \langle  \alpha,\nabla_{t_i}\eta_i\rangle
=  - \sum_i \langle \nabla_{t_i} \alpha,\eta_i\rangle=0.$$
So $\alpha$ is orthogonal to $PH^{d-1-p,p}(X)$. By Lemma \ref{ortho}, this
means that $\alpha$ is orthogonal to the whole group $H^{d-1-p,p}(X)$, hence
it is zero. Therefore   $H^{p,d-1-p}_T(X)=i^\ast(H^{p,d-1-p}(\var))$.
\end{proof}

For any variety $Y$  we define $H^{m,m}(Y,\Q) = H^{m,m}(Y,\C)\cap H^{2m}(Y,\Q)$.

\begin{lemma}\label{verygeneral} Let $d=2m+1\ge 3$, and assume that 
 the  hypotheses of the previous Lemma hold for $p=m$. Then for $z$ away from a countable union of subschemes of ${\mathcal Z}$ of positive codimension one has
$$H^{m,m}(X_z,\Q) = \operatorname{im}[ i^\ast\colon H^{m,m}(\var,\Q) \to H^{2m}(X_z.\Q)].$$
\end{lemma}
\begin{proof}
Let $\bar {\mathcal Z}$ be the universal cover of ${\mathcal Z}$. On it the (pullback of the) local system
$\mathscr H^{d-1}$ is trivial. Given a class $\alpha\in H^{m,m}(X)$
we can extend it to a global section of $\mathscr H^{d-1}$ by parallel transport using
the Gauss-Manin connection. Define the subset $\bar {\mathcal Z}_\alpha$ of $\bar {\mathcal Z}$ as the common zero locus of the sections $\pi_p(\alpha)$ of $\mathscr E^{p,d-1-p}$ for $p\ne m$ (i.e., the locus where $\alpha$ is of type $(m,m)$).

If $\bar {\mathcal Z}_\alpha=\bar {\mathcal Z}$ we are done because $\alpha$ is in
$H^{d-1}_T(X)$ hence is in the image of $i^\ast$ by the previous Lemma.
If $\bar {\mathcal Z}_\alpha \ne \bar {\mathcal Z}$, we note that $\bar {\mathcal Z}_\alpha$ is a subscheme of  $\bar {\mathcal Z}$.

We subtract from ${\mathcal Z}$ the union of  the projections of the
 subschemes $\bar {\mathcal Z}_\alpha$  where $\bar {\mathcal Z}_\alpha \ne \bar {\mathcal Z}$. The set of these varieties is countable because we are considering rational classes.
\end{proof}

\noindent {\em Proof of Theorem \ref{picard}}.  Lemma \ref{verygeneral}, for $d=3$, implies that
$H^{1,1}(X_z,\Q)$ and $H^{1,1}(\var,\Q)$  have the same dimension for a very general $z$. These  two numbers are the Picard numbers of $X_z$ and $\var$, respectively (see Definition/Proposition \ref{picards}). \qed

We assume now that $\var$ is Fano, and that $L =  -  K_{\var }$, so that the hypersurfaces
in the linear system $\vert L \vert $ are K3 surfaces.
We have   $PH^{2,0}\simeq R(f)_{0} \simeq \C$,
$PH^{1,1} (X) \simeq R(f)_{\beta}$, and $T_X\mathcal M_{\beta} \simeq R(f)_\beta$, where $\beta = - \deg K_{\var }$.
By Propositions \ref{isoring} and \ref{gammaisring}, the morphism $\gamma_2$ corresponds to the multiplication $R(f)_\beta\otimes R(f)_0\to R(f)_\beta$, and since $R(f)_0\simeq\C$, this is an isomorphism.
From Theorem \ref{picard} we have:

 \begin{thm}  Let $\var$ be a 3-dimensional Fano complete simplicial toric variety,  and $X$ a very general  hypersurface in the linear system $\vert -K_{\var}\vert$.
 Then $X$ has the same Picard number as $\var$.
\end{thm}
\

\bigskip

\bigskip\frenchspacing

\def\cprime{$'$}

\end{document}